\pdfoutput=1                      % per rendere testo %accessibile a studenti ipovedenti

\documentclass[10pt,a4paper]{amsart}

\usepackage[accsupp]{axessibility}    % per rendere testo accessibile a studenti ipovedenti

\usepackage{latexsym,amsfonts,amsmath,amsthm,amssymb, graphicx, mathrsfs, fancyhdr, 
amscd, 
fontenc,verbatim,stix,amsfonts,amsthm,indentfirst}

\numberwithin{equation}{section} %For Eq. Numb. (Sec.Numb)
\usepackage[all]{xy}
\usepackage{braket}

\usepackage{color,xcolor}

\usepackage{amssymb}

\usepackage[draft=false,setpagesize=false,pdfstartview=FitH,
colorlinks=true,citecolor=blue]{hyperref}
\theoremstyle{definition}
\newtheorem{definition}{Definition}[section]
\newtheorem{example}[definition]{Example}

\theoremstyle{plain}
\newtheorem{proposition}[definition]{Proposition}
\newtheorem{lemma}[definition]{Lemma}
\newtheorem{theorem}[definition]{Theorem}

\newtheorem{conjecture}{Conjecture}
\newtheorem*{conjecture*}{Conjecture}
\newtheorem{question}[conjecture]{Question}
\theoremstyle{definition}
\newtheorem{remark}[definition]{Remark}

\newcommand{\diver}{\mathrm{div}}

\DeclareMathOperator{\sh}{\mathrm{sh}}
\DeclareMathOperator{\ch}{\mathrm{ch}}

\newcommand{\UD}{\mathrm{UD}}
\newcommand{\BD}{\mathrm{BD}}

\newcommand{\cut}{\mathrm{cut}}

\newcommand{\ess}{\mathrm{ess}}

\newcommand{\eps}{\varepsilon}

\newcommand{\vol}{\mathrm{vol}}
\newcommand{\di}{\mathrm{d}}

\newcommand{\R}{\mathbb R}

\newcommand{\HH}{\mathbb H}

\newcommand{\lip}{\mathrm{Lip}}
\newcommand{\loc}{\mathrm{loc}}
\newcommand{\disp}{\displaystyle}

\newcommand{\gr}{\mathcal{G}}

\newcommand{\sgrh}{\mathscr{G}_{\Hc}}

\newcommand{\Ric}{\mathrm{Ric}}
\newcommand{\Sec}{\mathrm{Sec}}

\newcommand{\roy}{\mathsf{Ro}}

\newcommand{\Hc}{\mathrm{H}}

\makeatletter      %to define the Kulkarni-Nomizu product
\newcommand*\owedge{\mathpalette\@owedge\relax}
\newcommand*\@owedge[1]{
	\mathbin{
		\ooalign{
			$#1\m@th\bigcirc$\cr
			\hidewidth$#1\m@th\wedge$\hidewidth\cr
		}
	}
}
\makeatother

\newcommand{\supp}{\operatorname{supp}}

\newcommand{\Sph}{\mathbb{S}}

\allowdisplaybreaks[1]

\begin{document}

\author{Luciano Mari}
\address{Dipartimento di Matematica ``F. Enriques", Universit\`a degli studi di Milano, Via Saldini 50, I-20133 Milano, Italy.}
\email{luciano.mari@unimi.it}

\author{Marcos Ranieri}
\address{Instituto de Matem\'{a}tica, Universidade Federal de Alagoas, 57072-970 Macei\'{o}, Alagoas, Brazil}
\email{marcos.ranieri@im.ufal.br}

\author{Elaine Sampaio}
\address{Instituto de Matem\'{a}tica, Universidade Federal de Alagoas, 57072-970 Macei\'{o}, Alagoas, Brazil}
\email{elaine.carlos@im.ufal.br}

\author{Feliciano Vitório}
\address{Instituto de Matem\'{a}tica, Universidade Federal de Alagoas, 57072-970 Macei\'{o}, Alagoas, Brazil}
\email{feliciano@pos.mat.ufal.br}

\title{Spectral rigidity of manifolds with Ricci bounded below and maximal bottom spectrum}
\date{\today}
\maketitle    

\begin{abstract}
We investigate the spectrum of the Laplacian on complete, non-compact manifolds $M^n$ whose  Ricci curvature satisfies 
\(
\Ric \geq -(n-1)\Hc(r),
\)
for some continuous, non-increasing $\Hc$ with $\Hc-1 \in L^1(\infty)$. We prove that if the bottom spectrum attains the maximal value $\frac{(n-1)^2}{4}$ compatible with the curvature bound, 
%$\lambda(M)$ is maximal, that is, \(
%\lambda(M) = \frac{(n-1)^2}{4},
%\)
then the spectrum of $M$ coincides with that of $\mathbb{H}^n$, namely,
\(
\sigma(M) = \left[ \frac{(n-1)^2}{4}, \infty \right).
\)
The result can be localized to an end $E$ with infinite volume.
\end{abstract} 
\vspace{0.5cm}



\section{Introduction}

Let \( M \) be a complete, non-compact Riemannian manifold of dimension \( n \ge 2 \), and let \(\Delta \) denote the Laplace-Beltrami operator acting on \( C_c^\infty(M) \), with the agreement that $\Delta = \di^2/\di x^2$ on $\R$. It is known that \( \Delta \) admits a unique self-adjoint extension on a domain \(D(\Delta) \subset L^2(M)\). Hereafter, we denote by $\sigma(M) \subset [0,\infty)$ the spectrum of \( (-\Delta, D(\Delta)) \) and by $\lambda(M) \doteq \inf \sigma(M)$ the bottom spectrum of $M$. We will often refer to $\sigma(M)$ as the spectrum of $M$. We split $\sigma(M)$ into the essential spectrum \( \sigma_\ess(M) \), consisting of the accumulation points of $\sigma(M)$ and of the isolated eigenvalues of infinite multiplicity, and the discrete spectrum $\sigma_{\di}(M) = \sigma(M) \setminus \sigma_{\text{ess}}(M)$. By the decomposition principle, complete manifolds which are isometric outside a compact set have the same essential spectrum (as a set), thus $\sigma_\ess(M)$ only depends on the geometry at infinity of $M$. Eigenvalues $\lambda \in \sigma_\ess(M)$ are called embedded eigenvalues.

Thanks to the works of various authors from the end of the nineties to the beginning of 2010s, the spectrum of manifolds with asymptotically non-negative Ricci curvature is now well-understood as a set. Recall that $\sigma(\R^n) = [0,\infty)$ with no embedded eigenvalues. By works of Li \cite{Jli}, Donnelly \cite{donnelly}, Wang \cite{wang} and Lu \& Zhou \cite{lu_zhou} in increasing generality, the following holds:

\begin{theorem}\label{teo_luzhou}
	Let $M^n$ be a complete non-compact Riemannian manifold satisfying 
	\begin{equation}\label{eq_lowerricci}
		\liminf_{x \to \infty} \Ric_x = 0.
	\end{equation}
	Then, $\sigma(M)$ coincides with that of $\R^n$, i.e. $\sigma(M) = [0,\infty)$.
\end{theorem}

The proof hinges on a deep result in spectral theory by Sturm \cite{sturm}, see also Charalambous \& Lu \cite{CharalambousLu}. We stress that $\sigma(M) = [0,\infty)$ is just an identity between sets, with no claim about the corresponding spectral measures, in particular about the possible presence of embedded eigenvalues in $\sigma(M)$. 

\begin{quote}
	\emph{Hereafter, an identity $\sigma(M) = \sigma(N)$ between the spectrum (or, likewise, the essential spectrum) of two manifolds will be meant as subsets of $\R$, with no claim about the respective embedded eigenvalues.}
\end{quote}

\begin{remark}
	The techniques developed to prove Theorem \ref{teo_luzhou} are also effective for weighted manifolds with non-negative Bakry-Emery Ricci curvature to estimate the spectrum of the drifted Laplacian $\Delta_f$ (and, remarkably, of $\Delta$ itself) under natural assumptions on $f$. The interested reader is referred to \cite{CharalambousLu, lu_zhou} and to the work of Silvares \cite{silvares}.	
\end{remark}

The goal of this paper is to study $\sigma(M)$  when $\Ric$ is bounded below but does not satisfy \eqref{eq_lowerricci}. Precisely, we  will consider manifolds with
\begin{equation}\label{eq_assu_ricc}
\Ric \ge - (n-1)\Hc(r),
\end{equation}
for some continuous function $\Hc(r) \to 1$ as $r \to \infty$. Here, and hereafter in the paper, $r$ is the distance from a fixed origin $o \in M$ and the inequality is intended in the sense of quadratic forms. An example is given by the hyperbolic space of curvature $-1$, whose spectrum satisfies
\[
\sigma(\HH^n) = \left[ \frac{(n-1)^2}{4},\, \infty \right)
\]
and has no embedded eigenvalues. Since $\Hc(r) \to 1$ as $r \to \infty$, volume comparison together with Brooks' theorem \cite{brooks} imply that
\begin{equation}\label{eq_brooks}
\inf \sigma_\ess(M) \le  \frac{1}{4} \left[\liminf_{r \to \infty} \frac{\log |B_r|}{r}\right]^2 \le \frac{(n-1)^2}{4}.
\end{equation}

However, $\sigma_\ess(M)$ may be disconnected, i.e. gaps may occur in $\sigma_\ess(M)$. A striking example was found by Lott \cite{lott}: for each $\eps>0$, he constructed a complete manifold $M_\eps$ with sectional curvatures in $[-1-\eps,-1+\eps]$, finite volume and an infinite number of gaps in  $\sigma_\ess(M)$ diverging to $\infty$. This is particularly interesting if compared to the case of finite volume hyperbolic manifolds, whose essential spectrum always coincides with 
$\sigma(\HH^n)$. 

If $\vol(M) = \infty$, things are substantially different. For instance, Donnelly \cite{donnelly2} showed that a Cartan-Hadamard manifold $M$ (i.e. a complete, simply connected manifold with $\Sec \le 0$) whose sectional curvature tends to $-1$ at infinity satisfies 
$$
\sigma_{\ess}(M) = \sigma(\HH^n).
$$ 
Moreover, by a result due to McKean \cite{mckean}, $\sigma(M) = \sigma(\HH^n)$ if we further require $\Sec \le -1$ on the entire $M$. While Cartan-Hadamard $n$-manifolds are diffeomorphic to $\R^n$, \emph{conformally compact asymptotically hyperbolic manifolds} (CC-AH) provide a relevant class of spaces satisfying \eqref{eq_assu_ricc} that allow for a nontrivial topology. Recall that $(M,g)$ is said to be conformally compact if $M$ is diffeomorphic to the interior of a compact manifold $\overline{M}$ and there exists a defining function $\rho \in C^\infty(\overline{M})$ for $\partial \overline{M}$ such that  $\bar g = \rho^2 g$ is a smooth metric on $\overline{M}$. If $|\di \rho|_{\bar g} \equiv 1$ on $\partial \overline{M}$ we say that $(M,g)$ is CC-AH, since by explicit computation
\[
\Sec(\pi_x) = - 1 + O(\rho(x)) \qquad \forall \pi_x \subset T_xM. 
\]
In particular, this is automatically the case if $M$ is conformally compact and $\Ric = -(n-1)g$, i.e. $M$ is (Poincar\'e)-Einstein. Mazzeo \cite{mazzeo} proved that the essential spectrum of a CC-AH manifold satisfies $\sigma_\ess(M) = \sigma(\HH^n)$, with no embedded eigenvalues. If moreover $(M,g)$ is Poincar\'e-Einstein and the conformal infinity $(\partial \overline{M}, \bar g_{\partial \overline{M}})$ has non-negative Yamabe invariant, a theorem due to J. Lee \cite{lee} (see also \cite{x_wang} for an alternative proof) guarantees that $\lambda(M) \ge (n-1)^2/4$, whence
\[
\sigma(M) = \sigma(\HH^n). 
\]
On more general manifolds satisfying \eqref{eq_assu_ricc}, very little is known. Charalambous \& Lu ~\cite{CharalambousLu} conjectured that any complete manifold with $\Ric \ge -(n-1)$ and infinite volume should have an essential spectrum of the form \( [a, \infty) \) for some \( a \geq 0 \). However, Schoen \& Tran~\cite{TranSchoen} constructed examples of complete manifolds with bounded geometry (bounded curvature tensor, positive injectivity radius) whose essential spectrum has an arbitrary \emph{finite} number of gaps. Nevertheless, their construction gives no clear indication of the link between the curvature lower bound and the location of the spectral gaps. For instance, as far as we know, the following problems are open:

\begin{question}
	Prove or disprove: if a complete manifold $M^n$ satisfies $\Ric \ge-(n-1)$ and $\vol(M) = \infty$, then there exists $c = c(n,M)$ such that $[c,\infty) \subset \sigma(M)$.
\end{question}

\begin{question}
	Prove or disprove: if a complete, non-compact manifold $M^n$ satisfies $\Ric \ge-(n-1)$ and $\lambda(M)>0$, then there exists $c = c(n,\lambda)$ such that $[c,\infty) \subset \sigma(M)$.
\end{question}

This paper aims to answer Question 2 in the case of maximal $\lambda(M)$. 
%More precisely, we show that in this case the spectrum of $M$ has no gaps. 
To state the result in full generality, recall that an end $E \subset M$ is a connected component with non-compact closure of the complement of a compact set. The bottom spectrum $\lambda(E)$ is, by definition, the bottom of the spectrum of the Friedrichs extension of $(-\Delta, C^\infty_c(E))$, and it is variationally characterized as 
\[
\lambda(E) = \inf \left\{ \int_E |\nabla \phi|^2 \ : \ \phi \in C^\infty_c(E), \ \int_E \phi^2 =1 \right\}.
\]
Clearly, $\lambda(M) \le \lambda(E)$. By localizing \cite[Theorem 3.1]{docarmo_zhou} to a single end and using Bishop-Gromov's theorem, the bound \eqref{eq_assu_ricc} for some $\Hc(r)$ converging to $1$ as $r \to \infty$ implies 
\[
\lambda(E) \le \frac{(n-1)^2}{4} \qquad \text{for each end $E$ with } \, \vol(E) = \infty.
\]

Here is our main result: 

\begin{theorem} \label{mainthm} Let $(M^n,g)$ be a complete manifold such that 
	\begin{equation}\label{eq_Riccibound_intro}
	\Ric \geq -(n-1)\Hc(r),
	\end{equation}
	where $\Hc$ is a continuous, non-increasing function satisfying
	\begin{equation}\label{eq_inte_H}
	\int^\infty [\Hc(r)-1] \di r < \infty.
	\end{equation}
	If there exists an end $E$ such that
	\begin{equation}\label{eq_existence_E}
	\vol(E) = \infty, \qquad \lambda(E) = \frac{(n-1)^2}{4},
	\end{equation}
	then $\sigma_\ess(M) \supset \left[ \frac{(n-1)^2}{4}, \infty \right)$. In particular, if \eqref{eq_existence_E} is strengthened to 
	\begin{equation}\label{eq_lambda1M} 
	\lambda(M) = \frac{(n-1)^2}{4},
	\end{equation}
	then $\sigma(M) = \left[ \frac{(n-1)^2}{4}, \infty \right)$.
\end{theorem} 

\begin{remark}
It is known that \eqref{eq_lambda1M} implies $\vol(M) = \infty$, thus \eqref{eq_existence_E} holds for some end $E$. 
\end{remark}

\begin{remark}
	The above assumptions on $\Hc$ guarantee that $1 \le \Hc(r) \downarrow 1$ as $r \to \infty$. Hence, under \eqref{eq_existence_E} (respectively, \eqref{eq_lambda1M}), $E$ (resp. $M$) has the maximal bottom spectrum compatible with the Ricci curvature condition.
\end{remark}

\begin{remark}\label{rem_sullivan}
The class of manifolds enjoying  \eqref{eq_Riccibound_intro}, \eqref{eq_inte_H} and \eqref{eq_lambda1M} allow for various examples of interest. As stated above, one of them is Poincar\'e-Einstein manifolds whose conformal infinity has non-negative Yamabe invariant. Another example is given by hyperbolic manifolds $M = \HH^n \backslash \Gamma$, where $\Gamma$ is a torsion-free, discrete group of isometries of $\HH^n$ which is geometrically finite (i.e. the fundamental domain for $\Gamma$ has finitely many sides). In this setting, Sullivan \cite{sullivan} showed that  \eqref{eq_lambda1M} holds if and only if the Hausdorff dimension $\delta(\Gamma)$ of the limit set of $\Gamma$  satisfies
	\[
	\delta(\Gamma) \le \frac{n-1}{2}.
	\]
	Moreover, the request of geometric finiteness can be dropped in dimension $n=2$.\par
	Nevertheless, the assumptions in Theorem \ref{mainthm} force some topological and geometrical restrictions on $M$, as follows from works by Li \& Wang \cite{liwang_pos_1, liwang_pos_2} and Wang \cite{wang_X}: indeed, 
%	in some cases rigidity can occur. \tcr{Some properties of manifolds satisfying the assumptions in Theorem \ref{mainthm} follow from works by Li \& Wang \cite{liwang_pos_1, liwang_pos_2}: indeed, 
	\begin{itemize}
			\item by combining \cite[Theorem 1.4]{liwang_pos_1} with  comparison theory (see Proposition \ref{prop_pinching} below), if \eqref{eq_Riccibound_intro}, \eqref{eq_inte_H} and \eqref{eq_existence_E} hold the volume growth of balls $B_R$ with a fixed center must enjoy the two-sided bound
	\[
	C^{-1} \exp\big\{(n-1)R\big\} \le \vol(B_R) \le C \exp\big\{(n-1)R\big\} \qquad \forall \, R \ge 1,
	\]
	for some constant $C>0$ depending on $M$; 
	\item by \cite{liwang_pos_2}, a complete manifold with $\Ric \ge -(n-1)$ for which \eqref{eq_lambda1M} holds must have only one end unless $M = \R \times N$ with the warped product metric $\di t^2 + e^{2t} g_N$, for some compact manifold $(N,g_N)$ with non-negative Ricci curvature.
	\item by \cite{wang_X}, a complete manifold with $\Ric \ge -(n-1)$ which is the Riemannian covering of a compact manifold must be $\HH^n$ whenever \eqref{eq_lambda1M} holds.
	\end{itemize}
%	Also, we quote a theorem by Wang \cite{wang_X}, according to which the Riemannian universal covering $M \to N$ of a compact manifold with $\Ric_N \ge -(n-1)$ must be $\HH^n$ whenever \eqref{eq_lambda1M} holds. 
Apart from these results, we are aware of no further rigidity for manifolds satisfying the assumptions in Theorem \ref{mainthm}.  
\end{remark}

The integrability condition \eqref{eq_inte_H} is essential to obtain our estimates. Interestingly, it relates to a curvature threshold recently highlighted for the existence of embedded eigenvalues: consider a radially symmetric model given by $\R^n$ with metric, in polar coordinates centered at $0$,
\[
g_{\Hc} = \di r^2 + h(r)^2 g_{\mathbb{S}}, 
\]
where $g_{\mathbb{S}}$ is the round metric on $\mathbb{S}^{n-1}$ and  
\[
\qquad h \in C^2(\R^+_0), \qquad h>0 \ \ \text{on } \, \R^+, \qquad  h(0) = 0, \qquad h'(0)=1.
\]
Setting $\Hc(r) \doteq h''(r)/h(r)$, we call the model $M_\Hc$. The radial sectional curvature (the sectional curvature of planes containing $\nabla r$) of $M_\Hc$ is given by $-\Hc(r)$, while the Ricci curvature has eigenvalues 
\begin{equation}\label{eq_ricci_eigen}
-(n-1)\Hc(r) \qquad \text{and} \qquad -(n-1)\Hc(r) + (n-2)\left[ \Hc(r) + \frac{1 - h'(r)^2}{h(r)^2} \right]
\end{equation}
%\[
%- H(r) + (n-2)\frac{1- h'(r)^2}{h(r)^2}
%\]
of multiplicity, respectively, $1$ and $(n-1)$. It follows that $\Hc(r) \to 1$ as $r \to \infty$ implies 
\[
\sigma_\ess(M_{\Hc}) = \sigma(\HH^n) 
\]
and that, if $h(r) \to \infty$ as $r \to \infty$, $M_\Hc$ satisfies $\Ric \ge -(n-1)\hat\Hc(r)$ for some $\hat \Hc(r) \to 1$ at infinity\footnote{Both facts can be seen by noting that $\phi(r) = (\Delta r)/(n-1) = h'(r)/h(r)$ satisfies $\phi' + \phi^2 = \Hc$. Riccati comparison (see \cite[Proposition 2.12]{bmr}) and the behaviour of the explicit solutions to $\hat{\phi}' + \hat{\phi}^2 = 1$ imply that either $\phi(r) \to 1$ or $\phi(r) \to -1$ as $r \to \infty$. The statement $\sigma_\ess(M_\Hc) = \sigma(\HH^n)$ follows by applying \cite[Theorem 1.2]{kumura1}.}. However, $\sigma_\ess(M_\Hc)$ may or may not contain embedded eigenvalues, depending on the rate of decay of $|\Hc(r) -1|$. Recent works \cite{jitoliu,jitoliu_2,kumura1} made breakthrough progress on this problem. For any given countable, possibly dense subset $A \subset \sigma(\HH^n)$ and any given positive function $C(r)$ diverging as $r \to \infty$, Jitomirskaya and Liu in \cite{jitoliu} were able to produce the first example of a model $M_\Hc$ satisfying
\begin{equation}\label{eq_radial_asicurv}
\Hc(r) \to 1, \qquad  r \, |\Hc(r) - 1| \le C(r)
\end{equation}
whose essential spectrum has an embedded eigenvalue at each $\lambda \in A$. If $A$ is finite, they constructed an example with $C(r)$ uniformly bounded. On the other hand, Kumura \cite{kumura1} proved that if $C(r)$ vanishes at infinity and $h(r) \to \infty$ as $r \to \infty$ (the latter guarantees \cite[(1.3)]{kumura_ess})
%(even more, if it is bounded by a sufficiently small constant), 
then no embedded eigenvalue exist. Note that the function $\Hc(r) = 1 + (1+r)^{-1}$ barely fails to satisfy \eqref{eq_inte_H}, and we may therefore ask the following 

\begin{question}\label{Q_main}
Let $(M^n,g)$ be a complete manifold satisfying 
\eqref{eq_Riccibound_intro} and \eqref{eq_inte_H}. If
\[
	\lambda(M) = \frac{(n-1)^2}{4},
\]
can the spectrum $\sigma(M) = \left[ \frac{(n-1)^2}{4}, \infty \right)$ contain embedded eigenvalues?
\end{question}

Taking into account the known criteria for the absence of embedded eigenvalues in \cite{donnelly,donnelly_garofalo,kumura1}, to address the question it might be necessary to considerably improve our current understanding of the behaviour of the fake distance $b$ constructed below. This would be of independent interest.

\begin{example}
	Following Remark \ref{rem_sullivan}, let $M = \HH^n/\Gamma$ be a hyperbolic manifold  with $\Gamma$ torsion free, discrete and geometrically finite. As observed by Mazzeo \cite[p.26]{mazzeo}, if $M$ has infinite volume (which is the case if $\delta(\Gamma) < n-1$) then a result by Lax and Phillips \cite{lax_phi} guarantees the absence of  embedded eigenvalues in $\sigma_\ess(M)$. In particular, if $\delta(\Gamma) \le \frac{n-1}{2}$, Question \ref{Q_main} has a negative answer for such manifolds. 
\end{example}

\vspace{0.3cm}

We briefly discuss our strategy for Theorem \ref{mainthm}. By a classical criterion due to Weyl, to prove that $\lambda \in \sigma(M)$ one has to construct a family of functions $\{u_j\} \subset D(\Delta)$ such that 
\begin{equation}\label{eq_weyl}
\|\Delta u_j + \lambda u_j\|_2 = o\big( \|u_j\|_2 \big) \qquad \text{ as } \, j \to \infty. 
\end{equation}
Typically, each $u_j$ is produced as a function $\eta_j(b)$, where $b : M \to \R$ mimics the distance $r$ to $o$ or to a compact subset of $M$, and $\eta_j$ relates to the model $M_\Hc$ to which $M$ compares. To make the construction effective, one therefore needs estimates on $|\nabla b|$ and $\Delta b$, singled out in  \cite{elworthy_wang,donnelly}.

The choice $b=r$ is not possible unless $o$ has empty cut-locus, because the distribution $\Delta r$ is typically not  $L^2_\loc$ in neighbourhoods of $\cut(o)$. However, if one is able to uniformly bound $\|u_j\|_\infty$ from above and below (i.e. the candidate function $\eta_j$, which is usually the case), Charalambous-Lu's criterion  \cite{CharalambousLu} allows to replace  \eqref{eq_weyl} by the request that
\[
\|\Delta u_j + \lambda u_j\|_1 = o\big( \|u_j\|_2 \big) \qquad \text{ as } \, j \to \infty. 
\]
Although, in general, $r$ still cannot be used to produce $u_j$ (as $\Delta r \not \in L^1_\loc$ whenever $\mathscr{H}^{n-1}(\cut(o)) \neq 0$, see \cite{mante_masce_ural}), since $|\Delta r|$ is a locally finite measure one can approximate $r$ by convolution keeping track of the error in the $L^1$-norm. In the setting of Theorem \ref{teo_luzhou}  the procedure leads to effective estimates and allows to conclude, but under assumption \eqref{eq_Riccibound_intro} it seems much harder to control the error terms.

For this reason, we here choose a different function $b$, obtained by reparametrizing the minimal positive Green kernel $\gr$ of $-\Delta$ on $M$ with pole at $o$. Such a choice is not new, as it was extensively used by Colding and Minicozzi \cite{cm1,cm2} to study the geometry of manifolds with $\Ric \ge 0$ and allowed  Donnelly \cite{donnelly} to prove Theorem \ref{teo_luzhou} for manifolds with $\Ric \ge 0$ and maximal volume growth. However, to our knowledge it was never used before on manifolds whose Ricci curvature satisfies \eqref{eq_Riccibound_intro}. The existence of $\gr$ follows from Li \& Wang's theory in  \cite{liwang_pos_1}, who proved that an end $E$ satisfying \eqref{eq_existence_E} is necessarily non-parabolic. Assuming \eqref{eq_assu_ricc}, one can therefore define a \textit{fake distance} $b : M \to [0,\infty)$ as follows: letting $\sgrh(r)$ be the minimal, positive Green kernel of $M_\Hc$ with pole at the origin, $b(x)$ is implicitly defined by
$$
\mathcal{G}(x)= \sgrh\bigl(b(x)\bigr).
$$
Note that $b=r$ when $M = M_\Hc$ and $o$ is the origin of $M_\Hc$. Explicit computation gives

\[
\Delta b = \frac{v_\Hc'(b)}{v_\Hc(b)}|\nabla b|^2,
\]
where \( v_\Hc \) denotes the volume of geodesic spheres in \( M_\Hc \). Estimating $\Delta b$ therefore leads to control $|\nabla b|$ in some integral form, and to this aim the following gradient estimate in \cite[Theorem 2.19]{mrs} turns out to be crucial:
\[
|\nabla b| \le 1 \qquad \text{on } \, M.
\] 
More precisely, one needs to control integrals of the type
\begin{equation}\label{eq_inte_coresti}
\int_{\{s \le b \le t\} \cap E} \big(1-|\nabla b|^2\big) 
\end{equation}
for $0< s < t$, which is where assumption \eqref{eq_existence_E} comes into play. Although $b$ may not be an exhaustion function, the issue is treated via some results in potential theory, see Proposition \ref{prop_diver} and Theorem \ref{thm_test}. 

\begin{remark}
	It is interesting to compare the proof of Theorem \ref{mainthm} to that of \cite[Theorem 1]{bmmv}, where properly immersed minimal submanifolds $M^n \to \HH^{n+p}$ whose density grows subexponentially at infinity are shown to satisfy $\sigma(M) = \sigma(\HH^n)$. In place of $b$, the authors use the extrinsic distance function from a fixed point in $\HH^{n+p}$, and \eqref{eq_inte_coresti} is estimated by carefully combining the monotonicity formula (in a strengthened form) with the upper bound on the density function. In a certain sense, condition \eqref{eq_existence_E} plays the role of the low density growth condition in \cite{bmmv}, and some integral identities in \eqref{eq_int_nabla_b} that of the monotonicity formula.
\end{remark}

The structure of the paper is as follows. In Section \ref{sec_prelim}, we review some background material on Green kernels and the Royden algebra, that we use to address the possible non-properness of $b$. We then recall the construction of \( b \), along with key estimates and structural identities, and eventually, in Theorem \ref{thm_test}, we show that our candidates for $u_j$ lie in $D(\Delta)$. Section \ref{sec_proof} is devoted to the proof of Theorem \ref{mainthm}.

\section{Preliminaries}\label{sec_prelim}
\subsection{Green Kernel and Royden Algebra of $M$}
We refer the reader to \cite{prs,grigoryan}. A manifold $(M^n,g)$ is called non-parabolic if it admits a positive Green kernel $\gr$ for $-\Delta$. Considering one of the points $o$ as fixed, $\gr$ satisfies
\[
\mathcal{G}(x) > 0 \quad \text{for all } x\in M\setminus\{o\}, \qquad \Delta \mathcal{G} = -\delta_o \quad \text{in the sense of distributions},
\]
where the latter means 
\[
\int_M \langle \nabla \gr(o,y), \nabla \phi(y)\rangle \di y = \phi(o), \qquad \forall \, \phi \in \lip_c(M).
\]
Let $\{\Omega_j\}$ be an exhaustion of \(M\) by relatively compact, smooth open sets satisfying
\[
o \in \Omega_j \Subset \Omega_{j+1} \quad \text{and} \quad \bigcup_{j=1}^\infty \Omega_j = M.
\]
If $M$ is non-parabolic, a minimal such $\gr$ can be constructed as the monotone limit of kernels $\gr_j$ defined on $\Omega_j$ with Dirichlet boundary conditions. Indeed, the non-parabolicity of $M$ is equivalent to the fact that $\gr_j$ converge locally uniformly (hence, smoothly) on $M \backslash \{o\}$. By comparison, the limit $\gr$ does not depend on the chosen exhaustion. 
%The interested reader is referred to \cite{grigoryan} for a detailed treatment and various equivalent conditions to parabolicity. 
By definition, an end $E \subset M$ with respect to some compact set $K$ (say, with smooth boundary) is said to be non-parabolic if the double of $E$ is a non-parabolic manifold without boundary, see \cite[Section 7]{prs}. It turns out that $M$ is non-parabolic if and only if, having fixed one  such $K$, at least one of the ends with respect to $K$ is non-parabolic.    Li and Wang \cite{liwang_pos_1,liwang_pos_2} studied in depth the geometry of ends satisfying $\lambda(E) > 0$. In particular, the following dichotomy holds: either 
\begin{itemize}
	\item[-] $\vol(E)< \infty$ and exponentially decays, i.e. $\vol(E \setminus B_R) \le C \exp \{ -2 R \sqrt{\lambda(E)}\}$; 
	\item[-] $\vol(E) = \infty$, $E$ is non-parabolic and its volume grows exponentially: $\vol(E \cap B_R) \ge C \exp \{2 R \sqrt{\lambda(E)}\}$,
\end{itemize}
for some constant $C>0$ and for $R \ge 1$. In particular, in the second case the whole $M$ is non-parabolic.

\begin{remark}\label{rem_useful}
The kernel $\gr(x)$ diverges as $x \to o$, but may not tend to zero as $x \to \infty$. However, by fixing any relatively compact open set $B$ containing $o$ and setting $$a = \max_{\partial B} \gr,$$ the maximum principle implies that $\{\gr_j > a\} \subset B$ for each $j$, hence $\{\gr > a\} \subset B$.  Consequently, setting 
\[
\ell \doteq \limsup_{x\to\infty}\mathcal{G}(x),
\]
we have \(\ell < \infty\) and, for any \(a>\ell\), the superlevel set \(\{\mathcal{G} > a\}\) is compact.
\end{remark}

For each $0 \le s<t$, define
\[
l(t) \doteq \{\gr = t\}, \qquad L(s,t) \doteq \{s \le \gr \le t\},
\]
and likewise with a subscript $j$ when referring to $\gr_j$. For each $t$ such that $l_j(t)$ is smooth (a.e. $t$, by Sard's theorem), by integrating $\Delta \gr_j = -\delta_o$ against
\[
\phi_\eps = \min\left\{ \frac{t + \eps - \gr_j }{\eps},1\right\}_{+},
\]
letting $\eps \to 0$ and using the coarea formula one gets
\begin{equation}\label{eq_level_Gj}
\int_{l_j(t)} |\nabla\gr_j| = 1 \qquad \text{for a.e. $t \in \R^+$}. 
\end{equation}
Therefore, again by the coarea formula,
\begin{equation}\label{eq_level_Gj_2}
\int_{L_j(s,t)} |\nabla\gr_j|^2 = t-s \qquad \text{for a.e. every $0< s \le t$}.  
\end{equation}
The identities \eqref{eq_level_Gj} and \eqref{eq_level_Gj_2} also hold for $\gr$, with the same proof, provided that $\overline{L(s,t)}$ is compact. We extend these identities to the case of non-compact level sets by recalling a standard construction in potential theory. For more details, see for instance \cite{ns,changsario} and the thesis \cite{valto}.

By definition, the \emph{Royden algebra} of $M$, $\roy(M)$, is the set of all functions $w \in C(M) \cap W^{1,2}_\loc(M)$ satisfying
$$
D_{M}(w) \doteq \int_M |\nabla w|^2 < \infty,
$$
and whose representation in (any) chart is absolutely continuous in each variable separately (also called \emph{Tonelli functions}). Two topologies are naturally defined on $\roy(M)$:
\begin{enumerate}
	\item[1.] \textit{The (uniform-Dirichlet) topology $\tau_{\UD}$}  is induced by the norm 
	\[
	\|w\|_{\UD} = \|w\|_{L^\infty(M)} + \|\nabla w\|_{L^2(M)},
	\] 
	which makes $(\roy(M), \|\cdot\|_{\UD})$ into a Banach Algebra.
	\item[2.] \textit{The (bounded-Dirichlet) topology $\tau_{\BD}$} is defined via an (uncountable) basis of neighborhoods of \(0\) as follows: having fixed an exhaustion $\{\Omega_j\}$ of $M$,
		\[
		V(0, \{c_j\},\eps) = \Big\{ w \in \roy(M) \,:\, D_{M}(w) < \eps \;\text{and, for each $j$,}\; |w| < c_j \;\text{on}\; \Omega_j\Big\},
	\]
		for every \(\varepsilon>0\) and any divergent sequence \(\{c_j\}\) of positive numbers. In particular, a sequence $w_j \to w$ in $\tau_{\BD}$ if and only if
	\begin{equation}\label{propwj}
		\begin{aligned}
			&\text{(a)}\, \|w_j\|_{L^\infty(M)} \leq C < \infty, \quad \text{for some uniform constant } C; \\
			&\text{(b)}\, w_j \to w \;\text{ locally uniformly in } M; \\
			&\text{(c)}\, D_{M}(w_j-w) \to 0.
		\end{aligned}
	\end{equation}
\end{enumerate}
The algebra $\roy(M)$ contains two distinguished ideals: $\roy_c(M)$, the subspace of functions with compact support, and $\roy_0(M)$, defined as the \(\tau_{\BD}\)-sequential closure of $\roy_c(M)$. Moreover, the subalgebras $C^\infty(M) \cap \roy(M)$ and $C^\infty_c(M)$ are dense in $(\roy(M), \tau_{\UD})$ and $(\roy_0(M), \tau_{\BD})$, respectively. \par
Note that each point $x \in M$ defines a continuous functional $I_x$ on $(\roy(M),\tau_{\UD})$ by setting $I_x(f) = f(x)$. The functional is multiplicative, i.e. it satisfies $I_x(fg) = I_x(f)I_x(g)$, and it has unit norm.
\begin{definition}
	The \emph{Royden compactification} $M^*$ of $M$ is the set of unit norm multiplicative linear functionals (usually called \textit{characters}) on $(\roy(M), \tau_{\UD})$  with the weak* topology. 
	
\end{definition}
By Banach-Alaoglu's theorem, $M^*$ is compact, and moreover there is a continuous embedding $x \in M \mapsto I_x \in M^*$ into an open, dense subset of $M^*$. Each function $f \in \roy(M)$ can be continuously extended on $M^*$ by setting $f(\xi) \doteq \xi(f)$ for each character $\xi \in M^*$. The boundary $M^*\backslash M$ is characterized as the set 
\begin{equation}\label{defboundary}
	M^*\backslash M = \Big\{ \xi \in M^* \ \text{ such that } \, f(\xi)=0 \, \text{ for each } \, f \in \roy_c(M) \Big\}, 
\end{equation}

The behaviour at boundary points of functions in $\roy_0(M)$ leads to define the harmonic boundary $\partial_HM$ and the irregular boundary $\partial_I M$: 
$$
\begin{array}{lcl}
	\partial_H M & = & \disp \Big\{ \xi \in M^*\backslash M \ : \ \ f(\xi) = 0 \ \text{ for each } f \in \roy_0(M) \Big\}, \\[0.2cm]
	\partial_IM & = & \disp \Big\{ \xi \in M^*\backslash M \ : \ \ f(\xi) > 0 \ \text{ for some } f \in \roy_0(M) \Big\}.
\end{array}
$$ 
Moreover,  
\begin{itemize}
	\item[$(i)$] if $\Omega \subset M^*$ satisfies $\overline \Omega \cap \partial_H M = \emptyset$, then there exists $\eta \in \roy_0(M)$ such that $\eta \equiv 1$ on $\Omega$. 
	\item[$(ii)$] $\roy_0(M) \equiv \roy(M)$ if and only if $\partial_H M = \emptyset$, if and only if $M$ is parabolic.
	\item[$(iii)$] The ideal $\roy_0(M)$ coincides with the set 
	$\big\{ f \in \roy(M) : f(\partial_H M) =0 \big\}$.
\end{itemize}
Assume now that $M$ is non-parabolic, and let $\gr = \lim_j \gr_j$ be the minimal positive Green kernel of $\Delta$ centered at $o$. Direct computation shows that $\min\{\gr,t\} \in \roy_0(M)$. 

Since, by construction, 
$\overline{L(s,t)} \subset M^*$ has no intersection with the harmonic boundary $\partial_H M$, by (i) there exists a function $w \in \roy_0(M)$ which is identically $1$ on $L(s,t)$. By density, there exists a sequence $\{w_j\} \subseteq C^\infty_c(M)$ such that $w_j \to w$ in the $\BD$-topology.

We are ready to state the main result of this section, which improves on \cite[Lemma 5.1]{liwang3}.

\begin{proposition}\label{prop_diver}
	Let $M$ be a non-parabolic manifold, fix $o \in M$ and let $\gr$ be the minimal positive Green kernel with pole at $o$. Then: 
	\begin{itemize}
		\item[$(i)$] for almost every $0 < s < t < \infty$, 
		\begin{equation}\label{intfinitofront}
			\int_{l(t)}|\nabla \gr| = 1, \qquad \int_{L(s,t)} |\nabla \gr|^2 = t-s;
		\end{equation}
		\item[$(ii)$] for almost every $0<s<t< \infty$ and for every $u \in L^\infty(M) \cap \lip_\loc\big(\overline{L(s,t)}\big)$ satisfying
		\begin{equation}\label{intelocale}
			\int_{L(s,t)} |\nabla u|^2 < \infty,
		\end{equation}
		it holds
		\begin{equation}\label{teodiver}
			\int_{L(s,t)} \diver(u \nabla \gr) = \int_{l(t)} u |\nabla \gr| - \int_{l(s)} u|\nabla \gr|.
		\end{equation}
	\end{itemize}
\end{proposition}
\begin{proof}
	By the Coarea formula along the level sets of $\gr$, for every $\psi \in \lip_c(\overline{L(s,t)})$
	\begin{equation}\label{coarea1}
		\infty > \int_{L(s,t)} \psi^2|\nabla \gr|^2 = \int_s^t \left[\int_{l(z)}\psi^2|\nabla \gr|\right]\di z.
	\end{equation}
	Therefore, letting $\psi \uparrow 1$, by the monotone convergence theorem and since $\min\{\gr,t\} \in \roy(M)$, we deduce
	\[
	\int_{l(z)}|\nabla \gr| < \infty \qquad \text{for a.e. } z \in \R^+.
	\]
	Next, by Sard's theorem choose $s < t$ so that $l(s)$ and $l(t)$ are smooth hypersurfaces. Applying the divergence theorem and since the outward pointing normal on $l(t)$ is $\nu=\nabla \gr/|\nabla \gr|$ we get
	\begin{equation}\label{intgreenpr}
		\begin{array}{lcl}
			\disp \int_{L(s,t)}\psi^2 |\nabla \gr|^2 & = & \disp t\int_{l(t)}\psi^2(\partial_\nu \gr) - s\int_{l(s)}\psi^2(\partial_\nu \gr) - 2 \int_{L(s,t)} \psi \gr \langle \nabla \psi, \nabla \gr\rangle \\[0.4cm]
			& = & \disp t\int_{l(t)}\psi^2|\nabla \gr| - s\int_{l(s)}\psi^2|\nabla \gr| - 2 \int_{L(s,t)} \psi \gr \langle \nabla \psi, \nabla \gr\rangle.
		\end{array}
	\end{equation}
	Since $\overline{L(s,t)} \subset M^*$ is disjoint from the harmonic boundary, as explained above we can choose $w \in \roy_0(M)$ identically equal to $1$ on $L(s,t)$ and a sequence $\{w_j\} \subset C^\infty_c(M)$ satisfying \eqref{propwj}. Then, for $\psi=w_j$,
	\begin{equation}\label{essemplici}
		\left|\int_{L(s,t)} \psi \gr \langle \nabla \psi, \nabla \gr\rangle \right| \le Ct \int_{L(s,t)}|\nabla w_j||\nabla \gr| \le Ct \big[D_{L(s,t)}(w_j)\big]^{\frac{1}{2}} \big[D_{L(s,t)}(\gr)\big]^{\frac{1}{2}}.
	\end{equation}
	By \eqref{propwj} and since $w\equiv 1$ on $L(s,t)$, we have 
	\[
	D_{L(s,t)}(w_j) \le D_M(w_j - w) \to 0 \qquad \text{as } \, j \to \infty. 
	\]
	Letting $j \to\infty$ and applying Lebesgue Convergence Theorem to the boundary terms in \eqref{intgreenpr} with the aid of \eqref{intfinitofront}, we deduce the equality
	\begin{equation}\label{bueno!}
		\disp \int_{L(s,t)}|\nabla \gr|^2 = t\int_{l(t)}|\nabla \gr| - s\int_{l(s)}|\nabla \gr|.
	\end{equation}
	The locally integrable function
	$$
	f(s) = \int_{l(s)}|\nabla \gr| \qquad \text{thus satisfies}  \qquad \int_s^t f(z) \di z = t f(t) - s f(s),
	$$
	hence $tf(t)$ is differentiable almost everywhere, and differentiating the above identity with respect to $t$ we conclude that $f'(t)=0$. To show that $f(t) \equiv 1$, by Remark \ref{rem_useful} we choose $a>0$ so that $\{\gr > a\}$ has compact closure, whence $l(t)$ is relatively compact for each $t>a$. Choosing one such $t$ for which $l(t)$ is a smooth hypersurface, as shown earlier in this section, the identity
	\[
	\int_{l(t)} |\nabla \gr| = 1
	\]
	holds, thereby establishing the first part of \eqref{intfinitofront}. The second follows by the coarea formula.\\
	To show the validity of $(ii)$, we use again the approximating sequence $\{w_j\}$. Again by the divergence theorem with test function $\psi = uw_j \in \lip_c(\overline{L(s,t)})$,
	\begin{equation}\label{primadiver}
		\int_{L(s,t)} \diver(u w_j \nabla \gr) = \int_{l(t)} u w_j |\nabla \gr| - \int_{l(s)} u w_j |\nabla \gr|.
	\end{equation}
	Since $\gr$ is harmonic, the left-hand side can be rewritten as follows:
	\begin{equation}\label{passsatecni}
			 \int_{L(s,t)} \diver(u w_j \nabla \gr)  = \disp \int_{L(s,t)} u \langle \nabla w_j, \nabla \gr\rangle + \int_{L(s,t)} w_j \langle \nabla u, \nabla \gr\rangle.  
	\end{equation}

	By an estimate analogous to \eqref{essemplici}, since $u \in L^\infty(M)$ the first term goes to zero as $j\to \infty$. We now deal with the second one. From $$|w_j\langle \nabla u, \nabla \gr \rangle| \le C|\nabla u||\nabla \gr|,$$ and using \eqref{intelocale}, by Cauchy-Schwarz inequality we conclude that $|\nabla u||\nabla \gr| \in L^1(L(s,t))$ and we can apply Lebesgue Dominated Theorem to deduce
	\begin{equation}\label{off}
		\begin{array}{lcl}
			\disp \int_{L(s,t)} \diver(u w_j \nabla \gr) & = & \disp o(1) + 
			\int_{L(s,t)} w_j \langle \nabla u, \nabla \gr\rangle \\[0.4cm]
			& \to & \disp \int_{L(s,t)} \langle \nabla u, \nabla \gr\rangle = \int_{L(s,t)} \diver(u \nabla \gr),
		\end{array}
	\end{equation}
	where the last line again follows by the harmonicity of $\gr$. Finally, letting $j\to \infty$ in \eqref{primadiver} with the aid of \eqref{passsatecni} and \eqref{off}, and using $(i)$ to guarantee the convergence of the boundary integrals, we conclude the desired \eqref{teodiver}.
\end{proof}

\subsection{Fake distance and gradient estimates}\label{sectionfakedistance}

Given a non-parabolic manifold $(M^n,g)$ with Green kernel $\gr$ centered at $o \in M$, a \textit{fake distance} associated to $\gr$ can be defined by fixing a comparison model $M_{\Hc}$. The latter is constructed as follows: let $0 \le \Hc \in C(\R^+_0)$ and let $h \in C^2(\R^+_0)$ be the solution of 
\begin{equation}\label{eq_jacobi}
\left\{ 
\begin{array}{ll}
h'' - \Hc h = 0 & \quad \text{on } \, \R^+, \\[0.2cm]
h(0)=0, \ \ h'(0)=1. &	
\end{array}\right.
\end{equation}
Define $M_{\Hc}$ as being $\R^n$ with metric, in polar coordinates at the origin $0$, 
\[
g_{\Hc} = \di r^2 + h(r)^2 g_{\Sph},
\]
where $g_{\Sph}$ is the round metric of the unit sphere $\Sph^{n-1}$. For instance,   $M_1$  is  the hyperbolic space of curvature $-1$ and $h(r) = \sh(r)$. Note that the sectional curvatures of $M_\Hc$ satisfy
\[
\Sec(X \wedge \nabla r) = -\Hc(r), \qquad \Sec(X \wedge Y) = \frac{1-h'(r)^2}{h(r)^2}
\]
for each $X,Y$ perpendicular to $\nabla r$ and linearly independent.  The functions
\[
v_{\Hc}(r) = |\Sph^{n-1}| h(r)^{n-1}, \qquad V_{\Hc}(r) = \int_0^r v_{\Hc}(s)\di s
\]
describe, respectively, the volume of geodesic spheres and balls centered at $0$. By a result of Varapoulos \cite{varapoulos}, $M_{\Hc}$ is non-parabolic if and only if 
\begin{equation}\label{eq_model_nonparab}
	\int_1^\infty \frac{\di s}{v_{\Hc}(s)} < \infty.
\end{equation}
In this case, the minimal positive Green kernel $\sgrh$ of $M_{\Hc}$ centered at $0$ has the expression
\[
\sgrh(r) = \int_r^\infty \frac{\di s}{v_\Hc(s)},
\]
Note that \eqref{eq_model_nonparab} always holds if $\Hc \ge 1$ (which will be our case of interest), since by Sturm comparison $h(r) \ge \sh(r)$. Now, the fake distance $b: M \backslash \{o\} \rightarrow \mathbb{R}^{+}$ with origin $o$ and model $M_\Hc$ is implicitly defined via the identity
\begin{equation}\label{def_fake}
\gr(x) \doteq \sgrh\big(b(x)\big) = \int_{b(x)}^\infty \frac{\di s}{v_{\Hc}(s)}.
\end{equation}
We have that $b$ is well-defined, positive and smooth on $M \backslash \{o\}$. Moreover, since $\gr$ diverges as $x \to o$ then $b$ can be extended by continuity with $b(o) = 0$. Further, $b$ is proper if and only if $\gr(x) \to 0$ as $x$ diverges. 
%Observe that $b = r$, the distance function from $o$, if $M = M_{\Hc}$, hence the name ``fake distance". 
By rephrasing $(i)$ in Proposition \ref{prop_diver} in terms of $b$, we have the following identities:
\begin{equation}\label{eq_int_nabla_b}
	\int_{\{b=t\}} |\nabla b| = v_{\Hc}(t), \qquad \int_{\{ s \le b \le t\}} |\nabla b|^2 = V_{\Hc}(t) - V_{\Hc}(s),
\end{equation}
regardless of the properness of $b$. Moreover, differentiating $b$ twice, 
\begin{equation}\label{eq_delta_b}
\Delta b = \frac{v'_{\Hc}(b)}{v_{\Hc}(b)}|\nabla b|^2 = (n-1) \frac{h'(b)}{h(b)}|\nabla b|^2.
\end{equation}
Up until now, the model $M_{\Hc}$ need not have any link to $M$. Relating $M$ to $M_{\Hc}$ via curvature comparison, one is able to get sharp estimates for $b$; in particular, the following result was obtained by  Colding \cite[Theorem 3.1]{colding} for $\Hc \equiv 0$, and by the first author with Rigoli and Setti in \cite[Theorem 2.19]{mrs} in the form below, see also \cite{bmpr}.

\begin{theorem}\cite{mrs}\label{teo_gradesti}
	Let $(M^n,g)$ be a complete, non-compact, non-parabolic manifold and let $r, \gr$ be the distance from a fixed origin $o$ and the Green kernel with pole at $o$, respectively. Assume that 
	\[
	\Ric \ge - (n-1)\Hc(r)
	\]
	with $0 \le \Hc \in C(\R^+_0)$ non-increasing. If \eqref{eq_model_nonparab} holds, then the fake distance defined by \eqref{def_fake} satisfies 
	\[
	|\nabla b| \le 1 \qquad \text{on } \, M \backslash \{o\}. 
	\]
	Moreover, $|\nabla b(x_0)| =1$ for some $x_0 \in M\backslash \{o\}$ if and only if $M$ is isometric to $M_{\Hc}$. 
\end{theorem}

\begin{remark}
An analogous result holds when the fake distance is defined using the kernel of the $p$-Laplacian $\Delta_p$. In fact, as 
$p \to 1$, the corresponding fake distance construction provides a method to obtain solutions to the inverse mean curvature flow on $M$ (see \cite{mrs} for further details).
\end{remark}

The next result, a particular case of \cite[Proposition 1]{bmmv}, relates $M_\Hc$ to the hyperbolic space $M_1$ when $\Hc$ is close enough to $1$ at infinity.

\begin{proposition}\label{prop_pinching}
	Assume that $1 \le \Hc \in C(\R^+_0)$ satisfies 
	\[
	\Hc(s) - 1 \in L^1(\infty), 
	\]
	and define
	\begin{equation}\label{rel_useful}
		\zeta(s) \doteq \frac{h'(s)}{h(s)} - \frac{\ch(s)}{\sh(s)}.
	\end{equation}
	Then, 
		\begin{equation}\label{intebounds_zeta}
			\zeta(0^+) = 0, \quad \zeta \ge 0, \quad \zeta(s) \in L^1(\R^+), \quad \zeta(s) \to 0 \ \text{ as } \, s \to \infty,
		\end{equation}
and $h(s)/\sh(s) \to C$ as $s \to \infty$ for some $C \in \R^+$.
\end{proposition}

\noindent In particular, under the assumptions of Proposition \ref{prop_pinching}, we have
\begin{equation}\label{boundv}
C^{-1} \sh(r)^{n-1} \le v_{\Hc}(r) \le C \sh(r)^{n-1},  \quad 
\end{equation}
for some constant $C \in \R^+$.

\subsection{Laplacian and its domain}

Consider the Laplace-Beltrami operator $\Delta$, originally defined by
\[
\Delta\colon C_c^\infty(M)\to C_c^\infty(M),
\quad
u\mapsto \mathrm{div}(\nabla u)
\]
It is well known (see, e.g., \cite{Davies1989, ReedSimon1,kato}) that on a complete manifold, $\Delta$ is essentially self-adjoint. Consequently, the domain $D(\Delta)$ of its unique self-adjoint extension can be characterized in two equivalent ways:

\[
\begin{array}{lcl}
	D(\Delta) & = & \disp \Bigr\{
	u \in L^2(M) 
	\;:\;
	\Delta u \in L^2(M) \ \text{distributionally} \Bigr\} \\[0.4cm]
	& = & \disp \Bigl\{ u \in L^2(M) \ : \ \exists \{u_j\} \subset C^\infty_c(M) \ \  \text{ such that } \, \begin{array}{l} 
		u_j \to u \ \ \text{ in } \, L^2\\
		\{\Delta u_j\} \ \text{converges in $L^2$}
	\end{array} 
	\Bigr\},
\end{array}
\]
and $\Delta u = \lim_{j \to \infty} \Delta u_j$. By integrating $\Delta u_j$ against $\varphi^2 u_j$, where $\varphi$ a smooth cut-off and, applying Young's inequality, one obtains
\[
\frac{1}{2}\int_M \varphi ^2|\nabla u_j|^2\le -\int_M \varphi ^2u_j\Delta u_j+2\int_M u_{j}^2|\nabla \varphi |^2.
\]

Choosing the cut-offs
\[
\varphi(x)=\varphi_R(x)=\min\Bigl\{1,\frac{R-\di(x,B_R)}{R}\Bigr\}_+,
\]
and letting \(R\to\infty\) (followed by \(j\to\infty\)) shows that \(D(\Delta)\subset H^1(M)\) with
\[
\|\nabla u\|_2^2\le2\|u\|_2\,\|\Delta u\|_2.
\]

We recall Weyl’s characterization of the spectrum $\sigma(M)$:

\begin{lemma}{\cite[Lemma 4.1.2]{Davies1989}} A number $\lambda \in \mathbb{R}$ lies in $\sigma(M)$ if and only if there exists a sequence of nonzero functions $u_j \in D(\Delta)$ such that
	\[
	\| \Delta u_j + \lambda u_j \|_2 = o(\| u_j \|_2) \quad \text{as } j \to +\infty.
	\]
\end{lemma}
\noindent We call such $\{u_j\}$ a sequence of approximating eigenfunctions for $\lambda$. Note that Weyl's criterion does not require $u_j$ to have compact support, and this flexibility will be quite helpful to construct them. 

Assume that $M$ is complete and that an end $E \subset M$ satisfies   \begin{equation}\label{eq_end}
\vol(E) = \infty, \qquad \lambda(E) > 0.
\end{equation}
As mentioned above, by Li-Wang's theory \cite{liwang_pos_1} $M$ is non-parabolic. Let $\gr$ be a minimal positive Green kernel with pole at a fixed origin $o$, and let $r$ be the distance from $o$. 

\begin{theorem}\label{thm_test}
Let $M$ be a complete manifold satisfying 
\begin{equation}\label{eq_ipo2}
	\Ric \ge -(n-1)\Hc(r)
\end{equation}
for some $0 \le \Hc \in C(\R^+_0)$ non-increasing. Assume that there exists an end $E$ satisfying \eqref{eq_end}, and let $b$ be the fake distance with origin $o \not \in \overline{E}$ and model $M_\Hc$. Define $R_0 = \max_{\partial E} b$. Then, for every $\phi \in C^2_c((R_0,\infty))$ the function $\phi(b) \mathbb{1}_E$ belongs to $D(\Delta)$ and is not identically zero whenever $\phi \not \equiv 0$. 
\end{theorem}

\begin{proof}
	
	Observe that by construction $\phi(b) \mathbb{1}_E$ is a smooth function vanishing in a neighbourhood of $\partial E$. Moreover, it is not identically zero if so is $\phi$. Indeed, since $E$ is a non-parabolic end it is known that $\inf_E \gr = 0$ (it is enough to compare each $\gr_j$ to the harmonic function $h$ in \cite[Proposition 7.3]{prs} and to pass to limits), therefore $b(E) \supset [R_0,\infty)$, which is enough to conclude that $\phi(b) \mathbb{1}_E$ is nontrivial. We first prove that the sets $\{s \le b \le t\} \cap E$ have finite volume. For this purpose, we let $\gr_j$ be an increasing sequence of Green kernels corresponding to an exhaustion $\Omega_j \uparrow M$ with $\partial E \cap \Omega_j$, extended with zero away from $\Omega_j$. Fix $R \in (R_0,s)$ so that $\phi \in C^2_c((R,\infty))$ and choose $\psi \in C^\infty(M)$ vanishing in a neighbourhood of $M \backslash E$, satisfying $\psi \equiv 1$ in $E \cap \{b \ge R\}$ and with $|\nabla \psi|$ compactly supported. Denoting by $a = \max_{\partial E} \gr \ge \max_{E} \gr_j$, testing the Poincaré inequality with $\psi \gr_j \in \lip_c(E)$ and using Young's inequality we obtain 
	\[
	\lambda(E) \int_{\{ 0 < \mathcal{G}_j \le a \}} \psi^2\mathcal{G}_j^2 \leq 2\int_{\{ 0 < \mathcal{G}_j \le a \}} \psi^2|\nabla \mathcal{G}_j|^2 + \gr_j^2 |\nabla \psi|^2 \le 2a + a^2 \int_M|\nabla \psi|^2.
	\]
	Letting \( j \to \infty \), we find
	\[
	\lambda(E) \int_{\{ 0 < \mathcal{G} \le a \}} \psi^2\mathcal{G}^2 \leq 2a + a^2 \int_M|\nabla \psi|^2.
	\]
	Next, on the set $\{s \le b \le t\} \cap E$ it holds $\psi \equiv 1$ and $a_1 \le \gr \le a_2$ for some positive constants $a_1 \le a_2 < a$. Thus, 
	\[
	|\{s \le b\le t\} \cap E| \le |\{a_1 \le \gr \le a_2\} \cap E| \le a_1^{-2} \int_{\{ 0 < \mathcal{G} \le a \}} \psi^2\mathcal{G}^2 \le \frac{1}{a_1^2 \lambda(E)} \left[ 2a + a^2 \int_M|\nabla \psi|^2\right],
	\]
	as claimed.	Next, since $\phi$ has compact support in $(R_0,\infty)$, we fix $R_0<s<t$ so that ${\rm spt} \phi \subset [s,t]$. Therefore,  
	\[
	\int_E \phi(b)^2 \le \|\phi\|^2_\infty |\{ s \le b \le t\} \cap E| < \infty,
	\]
	hence $\phi(b)\mathbb{1}_E \in L^2(M)$. From 
	\begin{equation}\label{auxil_nabla_phi}
	\Delta \phi(b) = \phi''(b)|\nabla b|^2 + \phi'(b)\Delta b = \left[ \phi''(b) + \phi'(b)\frac{v'_\Hc(b)}{v_\Hc(b)}\right]|\nabla b|^2,
	\end{equation}
	since $ 0 \leq \Hc \leq \kappa$, for some constant $\kappa>0$, due to being non-increasing, Sturm's comparison Theorem guarantees that
	\[
	\frac{n-1}{b} = \frac{v'_0(b)}{v_0(b)} \le \frac{v'_\Hc(b)}{v_\Hc(b)} \le \frac{v'_\kappa(b)}{v_\kappa(b)} = (n-1)\sqrt{\kappa} \coth(\sqrt{\kappa}b). 
	\]
Besides, by Theorem \ref{teo_gradesti}, $|\nabla b| \leq 1$. Substituting this into $\eqref{auxil_nabla_phi}$ one obtains that there exists $C(s,\kappa,t)$ so that $|\Delta \phi(b)| \le C\|\phi\|_{C^2(\R)} \mathbb{1}_{\{s \le b \le t\} \cap E}$, and we deduce
	\[
	\Delta \big(\phi(b)\mathbb{1}_E\big) \in L^2(M).
	\]
	Thus, $\phi(b)\mathbb{1}_E \in D(\Delta)$.
	\end{proof}

\section{Proof of the Main Result}\label{sec_proof}

We begin with the following integral estimate, whose proof is inspired by \cite{liwang_pos_1}. For simplicity, we fix an end $E$, and for $0 \leq t < s$,  we define the set

$$
A_{t,s} \doteq \{ x \in E: t\le b(x) \le s\}
$$
In the assumptions on $M$ stated below, we saw that the fake distance $b$ can be constructed.

\begin{lemma}\label{mainlemma} Let $(M^n,g)$ be a complete manifold such that 
	\begin{equation}\label{eq_Riccibound}
	\Ric \geq -(n-1)\Hc(r),
	\end{equation}
where $1 \le \Hc \in C(\R^+_0)$ is non-increasing. If there exists an end $E$ such that
	\[
	\vol(E) = \infty, \qquad \lambda(E) = \frac{(n-1)^2}{4},
	\]
     then, for any $s > t > R_0 + 1$ with $R_0 = \mathrm{max}_{\partial E} b$, it holds
\begin{equation}\label{mainlemmaeq}
\int_{A_{t,s}} \gr^2 e^{(n-1)b} \big(1 - |\nabla b|^2\big) \, \le \, C, 
\end{equation}
where $C$ is a constant only depending on $n,R_0$. If furthermore
\begin{equation}\label{eq_inte_H_lem}
	\int^\infty [\Hc(r)-1] \di r < \infty,
\end{equation}
then it holds 
\begin{equation}\label{mainlemmaeq_2}
	\int_{A_{t,s}} e^{-(n-1)b} \big(1 - |\nabla b|^2\big) \, \le \, C, 
\end{equation}
for some $C = C(n,R_0,\Hc)$.
% In particular, if $$\lambda(M) = \frac{(n-1)^2}{4}$$ then,
%\begin{equation} 
%\int_{A_{t,s}} \gr^2 e^{(n-1)b} \big(1 - |\nabla b|^2\big) \, \le \, C.
%\end{equation} 
\end{lemma}

\begin{proof}[Proof of the Lemma \ref{mainlemma}]
%   First, since $M$ is complete, $\vol(E)=\infty$ and $\lambda(E)>0$, it follows from \cite[Theorem~1.4]{liwang_pos_1} that $E$ is a non-parabolic end. Hence $M$ itself is non-parabolic. Let $b$ be the fake distance associated to the minimal positive Green kernel $\mathcal{G}$ with pole at $o$, and set $R_0=\max_{\partial E} b$. 
Let $\phi \in C^2_c((R_0,\infty))$, and fix $[t_0,s_0] \subset (R_0, \infty)$ such that $\supp \phi \subset (t_0,s_0)$. Consequently, $\supp \phi(b) \subset \{ t_0 < b < s_0\}$ or, equivalently, $\supp \phi(b) \subset \{a_1 < \mathcal{G} < a_2\}= L(a_1,a_2)$ for positive constants $a_1,a_2$. Thus, 
\begin{equation}\label{eqdeftestpsi}
  \psi(x) \doteq \phi(b(x))\mathbb{1}_E(x)
  \end{equation}
is supported in $L(a_1,a_2) \cap E$. By Theorem \ref{thm_test} we have $\psi \in D(\Delta)$, and \eqref{teodiver} implies

\begin{equation}\label{auxl1}   
 \int_{M} \mathrm{div}( \psi^2 \mathcal{G} \nabla \mathcal{G}) = \int_{E \cap L(a_1,a_2)} \mathrm{div}( \psi^2 \mathcal{G} \nabla \mathcal{G}) = \int_{l(a_2) \cap E}\psi^2 \mathcal{G}|\nabla \mathcal{G}| - \int_{l(a_1) \cap E} \psi^2\mathcal{G}|\nabla \mathcal{G}| = 0.
 \end{equation}

Moreover, since $\mathcal{G}$ is harmonic and \eqref{auxl1} holds, we have

\begin{eqnarray} \label{eqinitialremark}
\int_{E} \left| \nabla \left( \psi \mathcal{G} \right) \right|^2 = \int_{M} \left| \nabla \left( \psi \mathcal{G} \right) \right|^2 \nonumber
%&=& \int_{L(a_1,a_2)} \left| \nabla \left( \phi \mathcal{G} \right) \right|^2 \nonumber \\ 
&=& \int_{M} \left| \nabla \psi \right|^2 \mathcal{G}^2 
+ 2 \int_{M} \psi \mathcal{G} \langle \nabla \psi, \nabla \mathcal{G} \rangle 
+ \int_{M} \psi^2 \left| \nabla \mathcal{G} \right|^2 \nonumber \nonumber \\
&=& \int_{M} |\nabla \psi|^2 \mathcal{G}^2 
+ \int_{M} \mathrm{div}( \psi^2 \mathcal{G} \nabla \mathcal{G}) \nonumber  \\
%&=& \int_{L(a_1,a_2)} |\nabla \psi|^2 \mathcal{G}^2 
%+ \int_{l(a_2)} \phi^2\mathcal{G}|\nabla \mathcal{G}| - \int_{l(a_1)} \phi^2\mathcal{G}|\nabla \mathcal{G}| \nonumber \\
%&=& \int_{L(a_1,a_2)} |\nabla \phi|^2 \mathcal{G}^2 \nonumber \\
&=&\int_{E} |\nabla \psi|^2 \mathcal{G}^2.
\end{eqnarray}

Choose $\phi(t) = \eta(t) e^{(n-1)t/2}$, where $\eta \in C^2_c((R_0,\infty))$ will be specified later. Define $\zeta(x) = \eta(b(x))$. By our assumption on $\lambda(E)$ and by \eqref{eqinitialremark}, the function \(\psi \mathcal{G}\) (with \(\psi\) as in \eqref{eqdeftestpsi}) satisfies
\[
\frac{(n-1)^2}{4}\int_E \zeta^2 e^{(n-1)b} \mathcal{G}^2 \leq \int_E \big|\nabla(\zeta e^{\frac{n-1}{2}b})\big|^2 \mathcal{G}^2.
\]
Expanding the right-hand side, we get

\begin{eqnarray*}
\frac{(n-1)^2}{4}\int_E \zeta^2 e^{(n-1)b} \mathcal{G}^2 &\leq& \int_E e^{(n-1)b} |\nabla \zeta|^2 \mathcal{G}^2 + (n-1) \int_E e^{(n-1)b} \zeta \langle \nabla \zeta, \nabla b \rangle \mathcal{G}^2 \\
& & + \frac{(n-1)^2}{4} \int_E \zeta^2 e^{(n-1)b} |\nabla b|^2 \mathcal{G}^2,
\end{eqnarray*}
equivalently,
\begin{equation}\label{aux_ineq1}
\frac{(n-1)^2}{4} \int_E \zeta^2 \mathcal{G}^2  e^{(n-1)b} \big(1 - |\nabla b|^2 \big)  
  \leq \int_E  e^{(n-1)b} \mathcal{G}^2 |\nabla \zeta|^2+(n-1) \int_E e^{(n-1)b} \zeta \langle \nabla \zeta, \nabla b \rangle \mathcal{G}^2.
\end{equation}
Choose $\eta$ satisfying
\[
0\le\eta\le1,\quad \eta\equiv0\text{ outside }(t-1,s+1),\quad \eta\equiv1\text{ on }(t,s),
\]
\[
|\eta'|\le 2 \, \text{ on } \, [t-1,t] \cup [s,s+1],
\] 
By \eqref{aux_ineq1} we have

\begin{eqnarray}\label{auxil2}
\nonumber\frac{(n-1)^2}{4} \int_{A_{t,s}} \mathcal{G}^2  e^{(n-1)b} \big(1 - |\nabla b|^2 \big)  
 &  \leq  & C\left( \int_{ A_{t-1,t}} \mathcal \gr^2 e^{(n-1)b}|\nabla b|^2 + \int_{ A_{s,s+1}} \mathcal \gr^2e^{(n-1)b}|\nabla b|^2 \right),
\end{eqnarray}
for some constant $C$ only depending on $n$. On the other hand, using the definition of $b$, the coarea formula and \eqref{eq_int_nabla_b}, for any $1 < a_1 < a_2$ we get

\begin{equation} 
\label{intcoarea} \int_{A_{a_1,a_2}} \mathcal{G}^2e^{(n-1)b}|\nabla b|^2 =  \int_{a_1}^{a_2} \sgrh^2(t)e^{(n-1)t}\left( \int_{\{b=t\} \cap E}|\nabla b|\right) \di t \le \int_{a_1}^{a_2} \sgrh^2(t) e^{(n-1)t}v_{\Hc}(t) \di t, \end{equation}
where we have used the identity \eqref{eq_int_nabla_b} in the last equality. Next, since $H \ge 1$, by Sturm comparison $h(t) \ge \sh(t)$ and thus
\[
\sgrh(t) \le \int_t^\infty \frac{\di s}{|\mathbb{S}^{n-1}|(\sh s)^{n-1}} \le Ce^{-(n-1)t},
\]
for some constant $C>0$ only depending on $n,R_0$. Furthermore, since $h(s)/\sh(s)$ is non-decreasing we can use \cite[Proposition 4.12]{bmr} to deduce that
\[
v_\Hc(t) \sgrh(t) = v_\Hc(t) \int_{t}^\infty \frac{\di s}{v_\Hc(s)} \le \sh^{n-1}(t) \int_{t}^\infty \frac{\di s}{(\sh s)^{n-1}} \le C.
\]
Summarizing,  

\begin{equation}
    \int_{a_1}^{a_2} \sgrh^2(t)e^{(n-1)t}v_{\Hc}(t) \di t \leq C(a_2-a_1), 
    %\qquad \gr^2 (t) \ge C^{-1} e^{-2(n-1)t} \quad \text{on } \, [R_0,\infty), 
\end{equation}

\noindent for some constant $\displaystyle C>0$ only depending on $n,R_0$. Inserting this into \eqref{auxil2} we obtain \eqref{mainlemmaeq}. If the integrability condition \eqref{eq_inte_H_lem} holds, we can use Proposition \ref{prop_pinching} to estimate
\[
\gr^2e^{(n-1)b} = \sgrh^2(b) e^{(n-1)b} = e^{(n-1)b} \left[\int_b^\infty \frac{\di s}{v_\Hc(s)}\right]^2 \ge C e^{(n-1)b} \left[\int_b^\infty \frac{\di s}{(\sh s)^{n-1}}\right]^2 \ge C e^{-(n-1)b},
\]
which yields \eqref{mainlemmaeq_2}.
\end{proof}

Now we are ready to prove the main result.

\begin{proof}[Proof of Theorem \ref{mainthm}]
    We aim to show that for every \( \displaystyle \lambda>\frac{(n-1)^2}{4}\), there exists a sequence \(\{u_j\} \subset C_0^\infty(M)\) satisfying
\begin{equation}\label{critweil}
\|\Delta u_j + \lambda u_j\|_2 \;=\; o\bigl(\|u_j\|_2\bigr)
\quad \text{as } j \to +\infty.
\end{equation}

\noindent To this end, consider
\( \displaystyle
\beta=\sqrt{\lambda-\frac{(n-1)^2}{4}}
\) and define
\[
\psi(s)=\frac{e^{i\beta s}}{\sqrt{v_{\Hc}(s)}}.
\]
Observe that \(\psi \) solves the differential equation
\begin{equation}\label{eq12}
\psi''+\frac{v'_{\Hc}}{v_{\Hc}}\psi '+\lambda\psi =a\psi ,
\end{equation}
where
\[
a(s)=\frac{(n-1)^2}{4}+\frac{1}{4}\biggl(\frac{v_{\Hc}'(s)}{v_{\Hc}(s)}\biggr)^2-\frac{1}{2}\frac{v_{\Hc}''(s)}{v_{\Hc}(s)}.
\]
From \eqref{eq_jacobi} and Proposition \ref{prop_pinching}, we have
\begin{equation}\label{eq_approx_basic}
\begin{array}{lcl}
	\disp \frac{v_{\Hc}' (s)}{v_{\Hc}(s)} & = & \disp (n-1) \coth s + o(1) = n-1 + o(1), \\[0.5cm]
	\disp \frac{v_{\Hc}''(s)}{v_{\Hc}(s)} & = & \disp (n-1)H + (n-1)(n-2) \left( \frac{h'}{h} \right)^2 \\[0.5cm]
	& = & \disp (n-1) + (n-1)(n-2) \coth^2 s + o(1) = (n-1)^2 + o(1), \\[0.3cm]
\end{array}
\end{equation}
whence \(a(s)\to0\) as \(s\to\infty.\) Next, fix $t_0 >>1$ and \(t,s, S \in \mathbb{R}\) such that
\[
t_0+1<t<s\le S-1,
\]
and consider a cutoff function \(\eta\in C_0^\infty(\mathbb{R})\) satisfying
\[
0\le\eta\le1,\quad \eta\equiv0\text{ outside }(t-1,S),\quad \eta\equiv1\text{ on }(t,s),
\]
together with the derivative bounds
\[
|\eta''|+|\eta'|\le C_0 \, \text{ on } \, [t-1,t]
\quad\text{and}\quad
|\eta''|+|\eta'|\le\frac{C_0}{S-s}\, \text{ on } \, [s,S],
\]
\noindent for some absolute constant $C_0$. Let \(b\) denote the fake distance function on \(M\). Define 
\[
  u_{t,s}(x) \;=\; \psi\bigl(b(x)\bigr)\,\eta\bigl(b(x)\bigr) \mathbb{1}_E(x).
\]
Then a straightforward computation shows

\begin{eqnarray*}
\Delta u_{t,s}+\lambda u_{t,s}
&=&\bigl(\eta''\psi+2\eta'\psi'+\eta\psi''\bigr)|\nabla b|^2
+(\eta'\psi+\eta\psi')\Delta b
+\lambda\eta\psi\\
&=&\bigl(\eta''\psi+2\eta'\psi'-\tfrac{v'}{v}\eta\psi'-\lambda\eta\psi+a\eta\psi\bigr)\bigl(|\nabla b|^2-1\bigr)
+a\eta\psi\\
&&+(\eta'\psi+\eta\psi')\Bigl(\Delta b-\frac{v'}{v}\Bigr)
+\Bigl(\eta''\psi+2\eta'\psi'+\eta'\psi\frac{v'}{v}\Bigr),
\end{eqnarray*}
where we used equation \eqref{eq12}. Note that there is an absolute constant \(c\) such that
\[
|\psi|+|\psi'|\le \frac{c}{\sqrt{v_{\Hc}}}.
\]
Thus, defining the weighted measure $\displaystyle \di \mu = \frac{\di x}{v_{\Hc}}$ (where $\di x$ is the Riemannian measure) and using the definition of $A_{a_1,a_2}$, we can write
\begin{equation}\label{eq13}
\displaystyle \|\Delta u_{t,s}+\lambda u_{t,s}\|_{2}^{2}
\leq \displaystyle
C\Biggl(
\int_{A_{t-1,S}}
\bigl[
(1-|\nabla b|^{2})^{2}
+\Bigl(\Delta b -\tfrac{v'_{\Hc}}{v_{\Hc}}\Bigr)^{2}
+a^2(b)
\bigr]
\,\di \mu
\;+\;
\frac{\mu(A_{s,S})}{(S-s)^{2}}
\;+\;
\mu(A_{t-1,t})
\Biggr),
\end{equation}
for some suitable constant $C(c,C_0)$. Using \eqref{eq_delta_b} and the first in \eqref{eq_approx_basic}, there exists a constant $C$ depending on $t_0$ but not on $s,t$ so that 

\[
\left(\Delta b-\frac{v_{\Hc}'}{v_{\Hc}}\right)^2
=\left(1-|\nabla b|^2\right)^{2}
\left(\frac{v_{\Hc}'}{v_{\Hc}}\right)^{2} \le C\left(1-|\nabla b|^2\right)^2.
\]

\noindent Setting
\[
F(t) \doteq \sup_{\sigma\in(t-1,\infty)}a^2(\sigma),
\]
we can rewrite \eqref{eq13} as
\begin{equation}\label{eq14}
\|\Delta u_{t,s}+\lambda u_{t,s}\|_{2}^2
\leq
C\Bigl(
\int_{A_{t-1,S}}(1-|\nabla b|^2)^2 \di \mu
+F(t)\,\mu(A_{t-1,S})
+\frac{\mu(A_{s,S})}{(S-s)^2}
+\mu(A_{t-1,t})
\Bigr).
\end{equation}
Since \(|\nabla b|^2\le 1\), the above estimate further simplifies to
\begin{equation}\label{eq15}
\|\Delta u_{t,s}+\lambda u_{t,s}\|_{2}^2
\leq
C\Bigl(
\int_{A_{t-1,S}}(1-|\nabla b|^2)\,\di \mu
+F(t)\,\mu(A_{t-1,S})
+\frac{\mu(A_{s,S})}{(S-s)^2}
+\mu(A_{t-1,t})
\Bigr).
\end{equation}
Under our hypotheses, Proposition~\ref{prop_pinching} applies directly. We will now bound each term appearing in the right‑hand side. To begin, Proposition~\ref{prop_pinching} yields
\[
C^{-1} e^{(n-1)b} \le v_\Hc(b) \le C e^{(n-1)b} \qquad \text{on } \, \{b \ge R_0\},
\]
for some constant $C$ depending on $n,R_0$, and thus 
\[
C^{-1} e^{-(n-1)b} \di x \le \di \mu \le C e^{-(n-1)b} \di x.
\]

Applying Lemma \ref{mainlemma} we therefore have

\begin{equation}\label{eqest1}
     \int_{A_{t-1,S}} (1 - |\nabla b|^2) \, \di \mu \le C \int_{A_{t-1,S}} e^{-(n-1)b} \big(1 - |\nabla b|^2\big)\di x \leq  C_2.
\end{equation}
Next, note that
\begin{equation}\label{eq_nablabyv}
    \|u_{t,s}\|_2^2 \ge \mu(A_{t,s}) 
    = \int_{A_{t,s}} \frac{\di x}{v_{\Hc}(b)} 
    \geq \int_{A_{t,s}} \frac{|\nabla b|^2}{v_{\Hc}(b)} \, \di x 
    = \int_t^s \frac{1}{v_{\Hc}(\sigma)} \left( \int_{\{b = \sigma\} \cap E} |\nabla b| \right) \di \sigma,
\end{equation}
\noindent where we used that $1 \geq |\nabla b|^2$ and the the coarea formula. Choose $R_0' \in (R_0, R_0+1)$ as a regular value of $\gr$, and for $\sigma \ge t$ define $a_1,a_\sigma$ to satisfy $\{a_\sigma \le \gr \le a_1\} =\{ R_0' \le b \le \sigma\}$. Applying Proposition \ref{prop_diver}, $(ii)$ with the choice $\bar u = \psi(b) \mathbb{1}_E$, where $\psi \in C^\infty_c(R_0,\infty)$ is identically $1$ on $[R_0',\sigma]$, we deduce
\[
\begin{array}{lcl}	
0 & = & \disp \int_{L(a_\sigma,a_1)} \diver( \bar u \nabla \gr) = \int_{l(a_1)} \bar u |\nabla \gr| - \int_{l(a_\sigma)} \bar u |\nabla \gr| \\[0.5cm]
& = & \disp \frac{1}{v_{\Hc}(R_0')} \int_{\{b=R_0'\} \cap E}|\nabla b| - \frac{1}{v_{\Hc}(\sigma)} \int_{\{b=\sigma\} \cap E}|\nabla b|.
\end{array}
\]
Plugging into \eqref{eq_nablabyv}, we deduce the existence of a constant $C_3$ independent of $t,s$ such that 
\begin{equation}\label{equ} 
|| u_{t,s}||_2^2 \geq C_3(s-t). 
\end{equation}
 
On the other hand, 
\begin{equation}\label{mubound} \mu(A_{t,s}) = \int_{A_{t,s}} \frac{\di x}{v_{\Hc}(b)} = \int_{A_{t,s}}|\nabla b|^2\frac{\di x}{v_{\Hc}(b)} + \int_{A_{t,s}}{(1-|\nabla b|^2)}\frac{\di x}{v_{\Hc}(b)} \leq s-t + C_2, \end{equation}

\noindent where we used \eqref{eqest1} and \eqref{eq_nablabyv}. Using the estimates $\eqref{eqest1}$ and $\eqref{mubound}$ and setting $S = s+1$, then \eqref{eq15} becomes

\begin{eqnarray*}
     \| \Delta u_{t,s}+\lambda u_{t,s}\|_{2}^{2} & \leq & C(C_2 + F(t)(s+1-(t-1)+ C_2) + (s+1-s + C_2) + (t-(t-1) + C_2)) \\
     & \leq & C_4 + C_4F(t)(s-t+C_3)
\end{eqnarray*}
where $C_4>0$ is a constant independent of $s$ and $t$. By the inequality $\eqref{equ}$, we have  $$ \frac{\| \Delta u_{t,s}+\lambda u_{t,s}\|_{2}^{2}}{|| u_{t,s} | | _2^2} \leq \frac{C_5}{s-t} + C_5F(t)\left(1 + \frac{C_3}{s-t}\right).
$$

\noindent Thus, since $\lim_{t \rightarrow \infty} F(t) = 0,$ we can conclude that

\begin{equation}\label{eq_liminf}
\liminf_{t \to \infty} \liminf_{s \to \infty} 
\frac{\|\Delta u_{t,s} + \lambda u_{t,s}\|_2^2}{\|u_{t,s}\|_2^2} = 0.
\end{equation}
We now construct a sequence of approximating eigenfunctions for \(\lambda\) as follows. Fix \(\varepsilon > 0\). By \eqref{eq_liminf}, there exists a divergent sequence \(\{t_i\}\) such that for all \(i \geq i_\varepsilon\),
\[
\liminf_{s \to +\infty} \frac{\| \Delta u_{t_i,s} + \lambda u_{t_i,s} \|_2^2}{\| u_{t_i,s} \|_2^2} < \frac{\varepsilon}{2}.
\]
For \(i = i_\varepsilon\), select a sequence \(\{s_j\}\) realizing this liminf. Then there exists \(j_\varepsilon = j_\varepsilon(i_\varepsilon, \varepsilon)\) such that for all \(j \geq j_\varepsilon\),
\begin{equation}\label{delta_u_varepsilon}
\| \Delta u_{t_{i_\varepsilon},s_j} + \lambda u_{t_{i_\varepsilon},s_j} \|_2^2 < \varepsilon \| u_{t_{i_\varepsilon},s_j} \|_2^2. 
\end{equation}
Define \(u_\varepsilon := u_{t_{i_\varepsilon},s_{j_\varepsilon}}\). By \eqref{delta_u_varepsilon}, from the family \(\{u_\varepsilon\}\) we can extract a sequence of approximating eigenfunctions for \(\lambda\), thus proving \(\lambda \in \sigma(M)\).

\end{proof}

\begin{remark}
	Theorem \ref{mainthm} relates to Elworthy $\&$ Wang's criterion in \cite[Theorem 1.1]{elworthy_wang} to guarantee that a half-line belongs to $\sigma_\ess(M)$. The criterion is stated in terms of an unbounded function $\gamma$ whose gradient and Laplacian are controlled in an integral sense. As a matter of fact, our main tool results (Theorem \ref{thm_test} and Lemma \ref{mainlemma}) and the gradient estimate $|\nabla b| \le 1$ in \cite{mrs} could be used to check that the choice $\gamma = b$ satisfies all the assumptions in \cite{elworthy_wang}. However, we preferred to include a direct argument in the proof of Theorem \ref{mainthm} for the sake of completeness and to point out some technical details. 
\end{remark}


\begin{thebibliography}{99}

\bibitem{bmr} B. Bianchini, L. Mari and M. Rigoli, \emph{On some aspects of oscillation theory and geometry.} Mem. Amer. Math. Soc. 225 (2013), no. 1056, vi+195 pp.

\bibitem{bmpr} B. Bianchini, L. Mari, P. Pucci and M. Rigoli, \emph{Geometric analysis of quasilinear inequalities on complete manifolds: maximum and compact support principles and detours on manifolds}. In Frontiers in Mathematics. Springer International Publishing, (2021).

\bibitem{brooks} R. Brooks, \emph{A relation between growth and the spectrum of the Laplacian}. Math. Z. 178(4) (1981), 501--508.

\bibitem{changsario}J. Chang and L. Sario, \emph{Royden's algebra on Riemannian spaces.}
Math. Scand. 28 (1971), 139--158.

\bibitem{CharalambousLu} N. Charalambous and Z. Lu, \emph{On the spectrum of the Laplacian}. Mathematische Annalen, 359(1) (2013), 211--238.

\bibitem{colding} T. H. Colding, \emph{New monotonicity formulas for Ricci curvature and applications. I.} Acta Math. 209(2) (2012), 229--263.

\bibitem{cm1} T. Colding and W. P. Minicozzi, \emph{Harmonic functions with polynomial growth.} J. Differential Geom. 46 (1997), no. 1, 1--77.

\bibitem{cm2} T. Colding and W. P. Minicozzi, \emph{Large scale behavior of kernels of Schr\"odinger operators.} Amer. J. Math. 119 (1997), no. 6, 1355--1398.

\bibitem{Davies1989} E. B. Davies, \emph{Heat Kernels and Spectral Theory.} Cambridge University Press, (1989).

\bibitem{docarmo_zhou} M. P. Do Carmo and D. Zhou, \emph{Eigenvalue estimate on complete noncompact Riemannian manifolds and applications.} Trans. Amer. Math. Soc. 351 (1999), no. 4, 1391--1401.

\bibitem{donnelly} H. Donnelly, \emph{Exhaustion functions and the spectrum of Riemannian manifolds}. Indiana University Mathematics Journal 46 (1997), no. 2, 505--527.

\bibitem{donnelly2} H. Donnelly, \emph{On the essential spectrum of a complete Riemannian manifold.} Topology 20.1 (1981), 1--14.

\bibitem{donnelly_garofalo} H. Donnelly and N. Garofalo, \emph{Riemannian manifolds whose Laplacians have purely continuous spectrum.}  Math. Ann. 293 (1992), no. 1, 143--161.

\bibitem{elworthy_wang} K. D. Elworthy and F.-Y. Wang, \emph{Essential spectrum on Riemannian manifolds.} Recent Developments in Stochastic Analysis and Related Topics (2004), 151--165.

\bibitem{grigoryan} A. Grigor'yan, \emph{Analytic and geometric background of recurrence and non-explosion of the Brownian motion on Riemannian manifolds.} Bull. Amer. Math. Soc. (N.S.) 36 (1999), no. 2, 135--249.

\bibitem{jitoliu} S. Jitomirskaya and W. Liu, \emph{Noncompact complete Riemannian manifolds with dense eigenvalues embedded in the essential spectrum of the Laplacian.} Geom. Funct. Anal. 29 (2019), no. 1, 238--257.

\bibitem{jitoliu_2} S. Jitomirskaya and W. Liu, \emph{Noncompact complete Riemannian manifolds with singular continuous spectrum embedded into the essential spectrum of the Laplacian, I. The hyperbolic case.} Trans. Amer. Math. Soc. 373 (2020), no. 8, 5885--5902.

\bibitem{kato} T. Kato, \emph{Perturbation theory for linear operators}. Springer Science and Business Media, (2013).

\bibitem{kumura_ess} H. Kumura, \emph{On the essential spectrum of the Laplacian on complete manifolds.} J. Math. Soc. Japan 49 (1997), no. 1, 1--14.

\bibitem{kumura1} H. Kumura, \emph{The radial curvature of an end that makes eigenvalues vanish in the essential spectrum I}. Math. Ann. 346 (2009), no. 4, 795--828.

\bibitem{kumura2} H. Kumura, \emph{The radial curvature of an end that makes eigenvalues vanish in the essential spectrum II}. Bulletin of the London Mathematical Society, 43 (2011),  no. 5, 985-1003.

%\bibitem{litam_harmonic} P. Li and L.-F. Tam. \emph{Harmonic functions and the structure of complete manifolds}. J. Differential Geom. 35 (1992), 359-383.

%\bibitem{Peter_Li_livro} P. Li. \emph{Geometric analysis}. Cambridge Studies in Advanced Mathematics 134, Cambridge University Press, 2012. MR 2962229 Zbl 1246.53002

\bibitem{lax_phi} P. Lax and R. Phillips, \emph{The asymptotic distribution of lattice points in euclidean and noneuclidean spaces.} J. Funct. Anal. 46 (1982), no. 3, 280--350.

\bibitem{lee} J. Lee, \emph{The spectrum of an asymptotically hyperbolic Einstein manifold.}
Comm. Anal. Geom. 3 (1995), no. 2, 253--271.

\bibitem{Jli} J. Y. Li, \emph{Spectrum of the Laplacian on a complete Riemannian manifold with nonnegative Ricci curvature which possess a pole.} J. Math. Soc. Japan 46 (1994), no. 2, 213--216.

\bibitem{liwang_pos_1} P. Li and J. Wang, \emph{Complete manifolds with positive spectrum.} J. Differential Geom. 58 (2001), no. 3, 501--534.

\bibitem{liwang_pos_2} P. Li and J. Wang, \emph{Complete manifolds with positive spectrum. II.} J. Differential Geom. 62 (2002), no. 1, 143--162.

\bibitem{liwang3} P. Li and J. Wang, \emph{Weighted Poincaré inequality and rigidity of complete manifolds}. Ann. Sci. Ec. Nor. Sup. 39 (2006), no. 6, 921--982.

\bibitem{bmmv} B. P. Lima, L. Mari, F. Montenegro and F. Vieira, \emph{Density and spectrum of minimal submanifolds in space forms}. Math. Ann. 366 (2016), no. 3-4, 1035--1066.

\bibitem{lott} J. Lott, \emph{On the spectrum of a finite-volume negatively-curved manifold}. Amer. J. Math. 123 (2001), no. 2, 185--205.

\bibitem{lu_zhou} Z. Lu and D. Zhou, \emph{On the essential spectrum of complete non-compact manifolds}. J. Funct. Anal. 260 (2011), no. 11, 3283-3298.

\bibitem{mante_masce_ural} C. Mantegazza, G. Mascellani and G. Uraltsev, \emph{On the distributional Hessian of the distance function.} Pacific J. Math. 270 (2014), no. 1, 151--166.

\bibitem{mrs} L. Mari, M. Rigoli and A. G. Setti,  \emph{On the 1/H-flow by p-Laplace approximation: new estimates via fake distances under Ricci lower bounds}. Amer. J. Math. 144 (2022), no. 3, 779--849. Corrigendum on: Amer. J. Math. 145 (2023), no. 3, 667--671.

\bibitem{mazzeo} R. Mazzeo, \emph{Unique continuation at infinity and embedded eigenvalues for asymptotically hyperbolic manifolds.} Amer. J. Math. 113 (1991), no. 1, 25--45.

\bibitem{mckean} H.P. McKean, \emph{An upper bound to the spectrum of $\Delta$ on a manifold of negative curvature.} J. Differential Geometry 4 (1970), no. 3, 359--366.

\bibitem{ns} M. Nakai and L. Sario, \emph{Classification theory of Riemann surfaces}. Springer Science \& Business Media, (2012) vol. 164.

\bibitem{prs} S. Pigola, M. Rigoli and A.G. Setti, \emph{Vanishing and finiteness results in geometric analysis. A generalization of the Bochner technique.} Progr. Math., 266 Birkh\"auser Verlag, Basel, (2008). xiv+282 pp.

%\bibitem{pw} P. Petersen and G. Wei. \emph{Relative volume comparison with integral curvature bounds}. GAFA 7 (1997), 1031-1045.

\bibitem{ReedSimon1} M. Reed and B. Simon, \emph{Methods of Modern Mathematical Physics, Vol. I: Functional Analysis.} Academic Press, (1972).

\bibitem{TranSchoen} R. Schoen and H. Tran, \emph{Complete manifolds with bounded curvature and spectral gaps}. J. Differential Equations, 261(2016), no. 4, 2584--2606.

\bibitem{silvares} L. Silvares, \emph{On the essential spectrum of the Laplacian and the drifted Laplacian.} J. Funct. Anal. 266 (2014), no. 6, 3906--3936.

\bibitem{sturm} K.T. Sturm, \emph{On the \( L^p \)-spectrum of uniformly elliptic operators on Riemannian manifolds}. J. Funct. Anal. 118 (1993), no. 2, 442--453.

\bibitem{sullivan} D. Sullivan, \emph{Related aspects of positivity in Riemannian geometry.} J. Differential
Geom. 25 (1987), no. 3, 327--351.

\bibitem{valto} D. Valtorta, \emph{Potenziali di Evans su variet\'a paraboliche.} (Italian). Master's Thesis, Universit\'a dell'Insubria. Available at arXiv:1101.2618.

\bibitem{varapoulos} N.T. Varopoulos,  \emph{Potential theory and diffusion on Riemannian manifolds}. In Conference on harmonic analysis in honor of Antoni Zygmund. Vol. 1 (1983), 821--837.

\bibitem{wang} J. Wang, \emph{The spectrum of the Laplacian on a manifold of nonnegative Ricci curvature}. Math. Res. Lett. 4 (1997), no. 4, 473--479.

\bibitem{x_wang} X. Wang, \emph{A new proof of Lee's theorem on the spectrum of conformally compact Einstein manifolds.} Comm. Anal. Geom. 10 (2002), no. 3, 647--651.

\bibitem{wang_X} X. Wang, \emph{Harmonic functions, entropy, and a characterization of the hyperbolic space.} J. Geom. Anal. 18 (2008), no. 1, 272--284.

%\bibitem{bdcs} P. B\'erard, M.P. do Carmo, and W. Santos, \emph{Complete Hypersurfaces with Constant Mean Curvature and Finite Total Curvature}. Ann. Glob. Anal. Geom. 16 (1998), 273-290.

%\bibitem{castillon} P. Castillon, \emph{Spectral properties of constant mean curvature submanifolds in hyperbolic space}. Annals of Global and Geometric Analysis (1999), 17: 563-579.

%\bibitem{Carron2020} G. Carron, \emph{Euclidean volume growth for complete Riemannian manifolds.} Milan J. Math. 88 (2020), no. 2, 455-478.


%\bibitem{Chavel1984} I. Chavel, \emph{Eigenvalues in Riemannian Geometry}, Academic Press, 1984.

%\bibitem{chz} L.-F. Cheung and D. Zhou, \emph{Stable constant mean curvature hypersurfaces in $\mathbb{R}^{n+1}$ and $\HH^{n+1}(-1)$}. Bull. Braz. Math. Soc. 36 (2005), no. 1, 99-114.

%\bibitem{cheng} S. Y. Cheng, \emph{Eigenvalue comparison theorems and its geometric applications.} Mathematische Zeitschrift, 143 (1975), 289-297.

%\bibitem{dcls} M.P. do Carmo, L.-F. Cheung, and W. Santos, \emph{On the compactness of constant mean curvature hypersurfaces with finite total curvature.} Arch. Math. 73 (1999), 216-222.

%\bibitem{munteanu2021positive} O. Munteanu, F. Schulze, and J. Wang, \emph{Positive solutions to Schrödinger equations and geometric applications}. Journal für die reine und angewandte Mathematik (Crelles Journal) 774 (2021), 185-217.

%\bibitem{prs} S. Pigola, M. Rigoli, and A.G. Setti, \emph{Vanishing and finiteness results in geometric analysis}. Progr. Math., 266, Birkh\"auser Verlag, Basel, 2008, xiv+282 pp.







\end{thebibliography}
\end{document}